\documentclass[oneside,leqno,10pt]{article}
\usepackage{amssymb,amsmath,latexsym,amsthm,stmaryrd}

\def\ga{\mathfrak{a}}

\def\gg{\mathfrak{g}}

\def\gk{\mathfrak{k}}

\def\gn{\mathfrak{n}}

\def\gs{\mathfrak{s}}

\def\gu{\mathfrak{u}}

\def\H{\mathbb{H}}

\def\Span{{\rm Span}\,}
\def\Ad{{\rm Ad}}

\def\Span{{\rm Span}\,}

\newtheorem{theorem}[equation]{Theorem}

\newtheorem{lemma}[equation]{Lemma}

\newtheorem{corollary}[equation]{Corollary}

\newtheorem{definition}[equation]{Definition}

\newtheorem{remark}[equation]{Remark}

\def\sideremark#1{\ifvmode\leavevmode\fi\vadjust{\vbox to0pt{\vss
 \hbox to 0pt{\hskip\hsize\hskip1em
\vbox{\hsize2cm\tiny\raggedright\pretolerance10000 
 \noindent #1\hfill}\hss}\vbox to8pt{\vfil}\vss}}} 

\title{Bounded Isometries and Homogeneous Quotients}

\author{Joseph A. Wolf\footnote{Research partially supported by the Simons
Foundation and by the Dickson Emeriti Professorship.\newline
}}
\date{November 25, 2015}

\begin{document}

\maketitle

\abstract{In this paper we give an explicit description of the bounded
displacement isometries of a class of spaces that includes the Riemannian 
nilmanifolds.  The class of spaces consists of metric spaces (and thus
includes Finsler manifolds) on which an exponential solvable Lie group
acts transitively by isometries.  The bounded isometries are proved to
be of constant displacement.  Their characterization gives further
evidence for the author's 1962 conjecture on homogeneous Riemannian
quotient manifolds.  That conjecture suggests that if $\Gamma \backslash M$
is a Riemannian quotient of a connected simply connected homogeneous
Riemannian manifold $M$, then $\Gamma \backslash M$ is homogeneous if and
only if each isometry $\gamma \in \Gamma$ is of constant displacement.
The description of bounded isometries in this paper gives an alternative 
proof of an old result of J. Tits on bounded automorphisms of semisimple 
Lie groups.
The topic of constant displacement isometries has an interesting history, 
starting with Clifford's use of quaternions in
non--euclidean geometry, and we sketch that in a historical note.}
\smallskip

\section{Introduction}\label{sec1}
\setcounter{equation}{0}
An isometry $\rho$ of a metric space $(M,d)$ is of {\em constant
displacement} if it moves each
point the same distance, i.e. if the displacement function
$\delta_\rho(x) := d(x,\rho(x))$ is constant.  
W. K. Clifford \cite{C1873} described such isometries for the $3$--sphere, 
using the then--recent discovery of quaternions.  Somewhat later
G. Vincent \cite{V1947} used the term ``Clifford translation''
for constant displacement isometries of round spheres in his study of 
spherical space forms $\Gamma \backslash S^n$ with $\Gamma$ metabelian.
Later the author (\cite{W1960}, \cite{W1961a}, \cite{W1961b}) used the
the term ``Clifford translation'' in the context of metric spaces,
especially Riemannian manifolds, proving

\begin{quote} 
\noindent {\bf Conjecture.}
Let $M$ be a connected, simply connected Riemannian homogeneous manifold
and let $M \to \Gamma \backslash M$ be a Riemannian covering.  Then 
$\Gamma \backslash M$ is homogeneous if and only if every $\gamma \in \Gamma$
is an isometry of constant displacement on $M$.
\end{quote}

\noindent
for the case where $M$ is a Riemannian symmetric space \cite{W1962b}.  In 
part the
argument was case by case, but later V. Ozols (\cite{O1969}, \cite{O1973},
\cite{O1974b}) gave a general argument for the
situation where $\Gamma$ is a cyclic subgroup of the identity component
$I^0(M)$ of the isometry group  $I(M)$.
H. Freudenthal \cite{F1963} discussed the situation where
$\Gamma \subset I^0(M)$, and introduced the term
{\em Clifford--Wolf isometry} (CW isometry) for isometries of constant 
displacement.  That seems to be the term in general usage.
\smallskip

Since then there has been a great deal of work on CW isometries and
their infinitesimal analogs, Killing vector fields of constant length,
in both the Riemannian and the Finsler manifold settings.
See \cite{BN2008a}, \cite{BN2008b}, \cite{BN2009},  
\cite{DMW1986}, \cite{DX2012}, \cite{DX2013a}, \cite{DX2013b}, \cite{DX2013c}, 
\cite{DX2014a}, \cite{DX2014b}, \cite{D1983}, and \cite{DW2012}.
The Conjecture was proved for Finsler symmetric spaces by S. Deng
and the author in \cite{DW2012}.
\smallskip

Most of the definitive results on CW isometries are concerned with Riemannian
(and later Finsler) symmetric spaces.  There we have a full understanding
of CW isometries (\cite{W1962b} and \cite{DW2012}).   The Conjecture is
known for CW isometries on some non--symmetric homogeneous Riemannian 
manifolds.   
The homogeneous Riemannian manifolds $(M,ds^2)$ for which the Conjecture is
known are: (i) Riemannian symmetric spaces \cite{W1962b}, (ii) Riemannian 
manifolds of non--positive sectional curvature \cite{W1964} and manifolds
without focal points \cite{D1983}, (iii) Riemannian
manifolds that admit a transitive semisimple group of isometries that has no
compact factor \cite{DMW1986}, (iv) Stieffel manifolds and some structurally
related compact homogeneous Riemannian manifolds (\cite{C1983}, \cite{C1986}),
(v) certain classes of
Riemannian normal homogeneous spaces (\cite{XW2014}, \cite{WPM2015}), 
and (vi) Riemannian nilmanifolds and Riemannian solvmanifolds (in this paper).
\smallskip

Here we give a complete structure theory for bounded isometries 
(isometries of bounded displacement) of metric spaces on which an 
exponential solvable Lie group acts transitively by isometries.  
We show that all bounded isometries are CW and belong to a certain
connected abelian group of CW isometries that is normal in the full
isometry group.  In the nilmanifold case that normal subgroup is the
center of the nilradical of the isometry group, but in other cases 
it may be smaller.  Since it is reduced to the identity
in the group $AN$ of an Iwasawa decomposition $G = KAN$, $G$ semisimple,    
this gives an alternative proof of J. Tits' theorem \cite{T1964}
that a semisimple Lie group with no compact factor has no nontrivial
bounded automorphism.
\smallskip

The class of spaces to which this applies includes Riemannian 
(and even Finsler) exponential solvmanifolds, in particular Riemannian 
nilmanifolds.  These results prove the Conjecture on homogeneous quotients
for those exponential solvmanifolds, and consequently Riemannian nilmanifolds.
\smallskip

In Section \ref{sec2} we work out a complete structure theory for
individual bounded isometries 
of metric spaces $(M,d)$ on which an exponential solvable Lie group 
$S$ acts transitively by isometries.  We first prove that the isometry
group $I(M,d)$ is a Lie group and that $I(M,d) = SK$ where $K$ is
an isotropy subgroup.  This is analogous to the Iwasawa decomposition
of a real reductive Lie group.  Then we show that every bounded isometry of
$(M,d)$ belongs to the center of $S$.  Thus every bounded
isometry is CW  and that center is
a normal subgroup of $I(M,d)$.  These results are in Theorem 
\ref{central} and its corollaries.
\smallskip

In Section \ref{sec3} we study quotients $\Gamma \backslash (M,d)$.
For locally isometric coverings $\psi: (M,d) \to \Gamma \backslash (M,d)$
we show that $\Gamma \backslash (M,d)$ is homogeneous if and only if
$\Gamma$ is a discrete group of CW isometries of $(M,d)$.  That
result is part of Theorem \ref{coverings}, which lists several other
equivalent conditions.  It proves the Conjecture for our class of metric 
spaces $(M,d)$, in particular for Riemannian (and Finsler) exponential
solvmanifolds.  One corollary is the infinitesimal version, for Riemannian 
(and Finsler) exponential solvmanifolds, characterizing the Killing vector 
fields of constant length. 
\smallskip

The arguments for Riemannian nilmanifolds are slightly less complicated
because a some technical considerations become transparent.
The nilmanifold version of our main result is
Corollary \ref{nil-coverings}.
\smallskip

The author thanks the referee for suggesting a better organization of this
paper and for suggesting that he include more explanatory background material.
\bigskip

\centerline{\Large \bf Historical Note}
\bigskip

The theory of constant displacement isometries can be traced back to
the independent discovery of quaternions by O. Rodrigues in 1840
\cite{R1840} and W. R. Hamilton in 1844 (\cite{H1844a}, \cite{H1844b}).
See \cite{A1989} for a description.  They used quaternions to describe
rotations of spheres, but W. K. Clifford \cite{C1873} seems to have introduced
their use in differential geometry in his construction\footnote{In modern 
terms, the group $\H'$ of unit 
quaternions, viewed as $S^3$, acts on itself by left and right
translations, $(a,b): q \mapsto aqb^{-1}$, and one can view the Clifford
torus as the orbit of $q$ as $a$ and $b$ each runs over a one--parameter 
subgroups of $\H'$.  Any two such one--parameter groups of transformations
of $S^3$ commute pointwise, so each such one--parameter group acts by
isometries of constant displacement on $S^3$.}  
of a flat torus in the sphere $S^3$.  In 1890 a paper of F. Klein \cite{K1890} 
introduced group theory {\em per se} into the picture. The next year
W. Killing introduced the term ``Clifford--Klein
space form'' for Riemannian manifolds of constant curvature 
and formulated the ``Clifford--Klein space form problem''
(\cite{K1891}, \cite{K1893}) in terms of quotients $\Gamma \backslash M$
where $M$ is a complete simply connected manifold of constant curvature.  
The classification of spherical space forms $\Gamma \backslash S^n$ 
was obscured in 1907 by the assertion\footnote{F. Enriques \cite[p. 117]{E1907}:
Endlich {\em l\" a\ss t sich 
eine dreidimensionale elliptische 
Raumform als Ganzes entweder auf den elliptischen oder auf den sph\" arischen
Raum in der Weise abwickeln, da{\ss}  jedem ihrer Punkte in diesem Raume eine
gewisse ganze Anzahl $p$ von (homologen) Punkten entspricht, wo zwei 
homologe Punkte durch eine sogennante Schiebung von der L\" ange 
$\frac{\ell\pi}{p}$ oder $\frac{2\ell\pi}{p}$ durch Deckung gebracht werden 
k\" onnen.} Dieses letzte Resultat erstreckt sich auf alle
elliptischen Raumformen von ungerader Dimensionenzahl $n$.} 
in the influential Enzyklop\" adie der Mathematischen
Wissenschaft that (in modern terms) if $\Gamma \backslash S^n$ is a
spherical space form, $n$ odd, then $\Gamma$ is a finite group of
constant displacement isometries of $S^n$.  This was corrected by 
H. Hopf in 1926 \cite{H1926} and by W. Threlfall and H. Seifert in 1930
(\cite{TS1930}, \cite{TS1932}) with the classification of all spherical 
space forms $\Gamma \backslash S^3$; see \cite{H1961}. This
was extended in 1947 by G. Vincent
\cite{V1947} for $\Gamma \backslash S^n$ with $\Gamma$ metabelian.
Vincent introduced the term ``Clifford translation'' and asked about their
relation to binary dihedral and polyhedral groups \cite[\S 10.4]{V1947}.
\smallskip

In 1960 the author formulated and proved the Conjecture for spaces of constant 
curvature \cite{W1960}, 
and in 1961 he used that result to answer Vincent's questions \cite{W1961a}.
In 1961 the author proved the Conjecture for Riemannian symmetric spaces
\cite{W1962b}.
Since then, as mentioned earlier in this Introduction, there has been a 
lot of progress toward the proof of the Conjecture, and this note is a
small step in that direction.

\section{Bounded Isometries inside Exponential Solvable Groups}\label{sec2}
\setcounter{equation}{0}

We will follow the convention that Lie groups are denoted by capital Latin 
letters and their Lie algebras are denoted by the corresponding lower
case German letters.  Thus, in the definition
\begin{definition}
A solvable Lie group $S$ is {\em exponential solvable} if the exponential
map $\exp: \gs \to S$ is a diffeomorphism.  Examples include the  simply
connected nilpotent Lie groups and the groups $AN$ in the Iwasawa
decomposition $G = KAN$ of a real semisimple Lie group.
\end{definition}
\noindent
it will be understood that $\gs, \gg, \gk, \ga \text{ and } \gn$ are the 
respective Lie algebras of $S, G, K, A \text{ and } N$.
\smallskip

We are looking at bounded isometries of metric spaces $(M,d)$ on which
an exponential solvable Lie group $S$ acts effectively and transitively 
(and thus, it will turn out, simply transitively) by isometries.  The
most interesting case is that of Riemannian exponential solvmanifolds.
By {\em Riemannian exponential solvmanifold} (relative to $S$) 
we mean a Riemannian manifold $M$ on which an exponential solvable Lie group 
$S$ of isometries acts transitively, and the kernel of the action of $S$
is discrete.  Then it is easy to see that the action of $S$ on $M$ lifts 
to a simply transitive action of the universal covering group of $S$ on 
the universal Riemannian covering space of $M$.  Examples
include connected simply connected Riemannian nilmanifolds and (see
\cite{W1964}, \cite{AW1976} and \cite{H1974}) connected simply connected 
Riemannian manifolds of non--positive sectional curvature.  However, 
except for the proof that the isometry group $I(M,d)$ is a Lie group,
the arguments are the same for metric spaces as for Riemannian manifolds,
so we work in that more general class.

\begin{lemma}\label{simply}
Let $(M,d)$ be a metric space on which an exponential solvable Lie group
$S$ acts effectively and transitively by isometries.  Then the
action of $S$ on $M$ is simply transitive.
\end{lemma}

\begin{proof} Let $x_0 \in M$.  The isotropy subgroup 
$S_{x_0} = \{s \in S \mid s(x_0) = x_0\}$ preserves all metric balls
$B_r(x_0) = \{x \in M \mid d(x,x_0) \leqq r\}$. As $S_{x_0}$ is a closed
subgroup, $S$ is a Lie group, and the $B_r(x_0)$ are compact, it follows 
from \cite{vDvW1928} that $S_{x_0}$ is compact.  
By definition of exponential solvable group, the only
compact subgroup of $S$ is $\{1\}$.
\end{proof}

For the rest of the section we fix a metric space  $(M,d)$ and an
exponential solvable Lie group $S$ acting effectively and
transitively by isometries.  We may view $(M,d)$ as the group manifold
$S$ with a left--invariant metric space structure.  The most interesting
cases are when $d$ is the distance function of a 
Riemannian metric $ds^2$ or a Finsler metric $F$.
In any case we will write $I(M,d)$ for the group of all isometries 
of $(M,d)$.
\smallskip

We will need an obvious elementary estimate; it is included for completeness.

\begin{lemma}\label{unbounded-unip-orb}
Let $U$ be a unipotent group of linear transformations of a real vector
space $V$.  Suppose that $v\in V$ is not a fixed point of $U$.  Then
$U(v)$ is unbounded, in other words is not contained in a compact
subset of $V$.
\end{lemma}

\begin{proof} Let $\xi \in \gu$ with $\xi(v) \ne 0$.  Then
$\exp(t\xi)v = \sum_0^r \frac{1}{n!}t^n\xi^n(v)$ where
$\xi^rv \ne 0 = \xi^{r+1}v$.  As $t \to \infty$ the
$\frac{1}{r!}t^r \xi^rv$ summand dominates the others and is unbounded.
\end{proof}

That is sufficient for our needs if $S$ is nilpotent.  But in general
we need a sightly less obvious version.

\begin{lemma}\label{unbounded-esolv-orb}
Let $S$ be an exponential solvable Lie group and $\xi \in \gs$ a
non--central element of the Lie algebra.  Then $\Ad(S)\xi$ is unbounded.
\end{lemma}

\begin{proof} Let $U$ be the nilradical of $S$.  If $U$ does not
centralize $\xi$ then $\Ad(S)\xi$ is unbounded by Lemma
\ref{unbounded-unip-orb}.  Thus we may assume $\Ad(U)\xi = \{\xi\}$.
If $s \in S$ now $\Ad(U)\Ad(s)\xi = \Ad(s) \Ad(U)\xi = \Ad(s)\xi$,
so $\Ad(U)$ acts trivially on $W:=\Span \Ad(S)\xi$.
\smallskip

The restriction $\Ad(\gs)|_W$ is a commutative linear Lie algebra in which
every nonzero element has an eigenvalue with nonzero real part.  Thus
$\Ad(S)\eta$ is unbounded for some $\eta \in W$, and consequently for some
$\eta \in \Ad(S)\xi$.  Now $\Ad(S)\xi$ is unbounded.
\end{proof}

\begin{theorem}\label{central}
Let $(M,d)$ be a metric space on which an exponential solvable Lie group
$S$ acts effectively and transitively by isometries.  Let $G = I(M,d)$.
Then $G$ is a Lie group, any isotropy subgroup $K$ is compact, and $G = SK$.  
If $g \in G$ is a bounded isometry then $g$ is a central element in $S$.
\end{theorem}

\begin{proof}
As noted above, $M$ carries a differentiable manifold structure for
which $s \mapsto s(x_0)$ is a diffeomorphism $S \cong M$.  As usual
$G = I(M,d)$ carries the compact--open topology.  The famous theorem of
van Danzig and van der Waerden \cite{vDvW1928} (or see 
\cite[Corollary 4]{M2010} for an exposition) says that $G$ is
locally compact and that its action on $M$ is proper.  In particular,
if $x_0 \in M$ then the isotropy subgroup 
$K = \{k \in G \mid k(x_0) = x_0\}$ is compact.  Further 
\cite[Corollary in \S6.3]{MZ1955} $G$ is a Lie group. Now 
$S$ and $K$ are closed subgroups, $G = SK$, and $M = G/K$.  
\smallskip

Express $g = sk$ with $s \in S$ and $k \in K$.  If $s = 1 \ne k$
then the differential of $k$ is unbounded on the tangent space to
$M$ at $x_0$\,, thus unbounded $\gs$, and thus unbounded on $(M,d)$.  
Thus $s \ne 1$ unless, of course, $g = 1$.
\smallskip

Suppose $s \ne 1$.  As $K$ is compact and the displacement function
$x \mapsto \delta_g(x)$ is bounded, $\Ad(G)g$ is bounded in $G$.
If $g' = s'k' \in G$ 
we compute $\Ad(g')g = s'k'skk'^{-1}s'^{-1}$.  That is bounded as 
$g'$ ranges over $G$, so $\Ad(S)s$ is bounded.  Let $N$ be the nilradical
of $S$.  Now $\Ad(N)s$ is a bounded unipotent $\Ad(N)$--orbit on $\gs$,
which is impossible unless $s$ centralizes $N$.  As in the proof of
Lemma \ref{unbounded-esolv-orb} it follows that $s$ is central in $S$.
\smallskip

Identify $\gs$ with the tangent
tangent space to $M$ at $x_0$\,.  Suppose
$\nu \in \gs$ with $\Ad(k)\nu \ne \nu$ and let $C$ be a compact
neighborhood of $0$ in $\gs$.  As $t \to \infty$,  $\Ad(k)(t\nu)$
must exit $t\nu + C$, so $\Ad(k)$ is unbounded on $\gs$.  
Thus $k$ is unbounded on $(M,d)$.  That is a contradiction, so
$\Ad(k)\nu = \nu$ for all $\nu \in \gs$, in other words $k$ is
trivial on the tangent space to $M$ at $x_0$\,.  As $M = I(M,d)/K$
it follows that $k = 1$.
\smallskip

Summarizing, the bounded isometry $g$ of $(M,d)$ is a central element
of $S$.
\end{proof}

\begin{corollary}\label{bounded-is-CW}
Let $(M,d)$ be a metric space on which an exponential solvable Lie group
$S$ acts effectively and transitively by isometries.  Then
every bounded isometry of $(M,d)$ is CW.
\end{corollary}

\begin{proof} Each bounded isometry $g$ is centralized by $S$, which is
transitive on $(M,d)$, so $g$ is CW \cite{W1960}.
\end{proof}

\begin{corollary}\label{center-normal}
Let $(M,d)$ be a metric space on which an exponential solvable Lie group
$S$ acts effectively and transitively by isometries.  Then the center of
$S$ consists of all the CW isometries of $(M,d)$, and it 
is an abelian normal subgroup of $I(M,d)$.
\end{corollary}

\begin{proof}  If $g \in I(M,d)$ is central in $S$ then it commutes with
every element of the transitive group$S$ of isometries, so \cite{W1960}
it is a CW isometry.  If $g \in I(M,d)$ is CW then Theorem \ref{central}
shows that it is a central element of $S$.
\end{proof}

A {\em Riemannian nilmanifold} is a connected Riemannian manifold $(M,ds^2)$
on which a nilpotent group $N$ of isometries acts transitively.  Then
\cite[Theorem 4.2]{W1963} $N$ is the nilpotent radical of the isometry 
group $I(M,ds^2)$, and $I(M,ds^2)$ is the semidirect product 
$N\rtimes K$ where $K$ is the isotropy subgroup at a point of $M$.
We abbreviate this situation by writing $G = I(M,ds^2) = N\rtimes K$,
so $M = G/K = (N\rtimes K)/K$.
\smallskip

In the case of Riemannian nilmanifolds Lemma \ref{simply} is obvious, 
Lemma \ref{unbounded-unip-orb}
is needed as stated, and the proof of Lemma \ref{unbounded-esolv-orb} is 
reduced to its first two sentences.  We can skip the first paragraph of the
proof of Theorem \ref{central}.  Corollaries \ref{bounded-is-CW} is
unchanged, and Corollary \ref{center-normal} is obvious, with $S = N$
nilpotent.
\smallskip

The alternative proof of a result of J. Tits, described in the Introduction, 
follows from Theorem \ref{central} and the Iwasawa decomposition 
$G = NAK$ of a real reductive Lie group $G$.  There $M$ is a Riemannian 
symmetric space of noncompact type, $AN = NA$ is exponential solvable, and 
$AN$ acts transitively by isometries on $M$.

\section{Homogeneous Quotients}\label{sec3}
\setcounter{equation}{0}
We apply Theorem \ref{central} to the structure of covering spaces 
$\psi: (M,d) \to \Gamma \backslash (M,d)$ where $(M,d)$ is a metric space 
on which an exponential solvable Lie group $S$ acts effectively and 
transitively by isometries.  Recall here \cite{W1960} that if
$\Gamma \backslash (M,d)$ is homogeneous, then $\Gamma$ consists of 
CW isometries.  We then indicate the simplification for Riemannian nilmanifolds.
The following is immediate from Theorem \ref{central}.

\begin{theorem}\label{coverings}
Let $(M,d)$ be a  metric space on which an exponential solvable Lie group 
$S$ acts effectively and transitively by isometries.  Let $x_0 \in M$,
let $G = I(M,d)$,  and
let $K$ denote the isotropy subgroup of $G$ at $x_0$\,.  Consider a 
locally isometric covering space 
$\psi: (M,ds^2) \to \Gamma \backslash (M,ds^2)$.
Then the following conditions are equivalent.
\smallskip

{\rm 1.} $\Gamma$ consists of bounded isometries of $(M,d)$.
\smallskip

{\rm 2.} $\Gamma$ consists of CW isometries of $(M,d)$.
\smallskip

{\rm 3.} The group $\Gamma$ is a discrete central subgroup $S$.
\smallskip

{\rm 4.} $\Gamma \backslash (M,ds^2)$ is a homogeneous metric space.
\smallskip

{\rm 5.} $\Gamma \backslash (M,ds^2)$ is a metric space on which a
Lie group $S/\Gamma$ acts transitively, where $S$ is exponential 
solvable and $\Gamma$ is a discrete central subgroup of $S$.
\end{theorem}

\begin{remark}
{\rm  Theorem \ref{coverings} applies in particular to Riemannian coverings
$\psi: (M,ds^2) \to \Gamma \backslash (M,ds^2)$, to Finsler manifold
coverings $\psi: (M,F) \to \Gamma \backslash (M,F)$, and to nilmanifolds.
Thus it tells us how to construct all connected Riemannian
nilmanifolds.  Start with a connected simply connected Lie group $N$, say
with center $Z$, and a discrete central subgroup $\Gamma \subset Z$.  Fix
a positive definite bilinear form $b$ on the Lie algebra $\gn$, and let 
$\overline{K}$ denote the group of all automorphisms of $N$ that normalize
$\Gamma$ and preserve  $b$.  Then $b$ translates around to define an
$((N/\Gamma)\rtimes \overline{K})$--invariant Riemannian metric
$dt^2$ on 
$\Gamma \backslash M=((N/\Gamma)\rtimes \overline{K})/\overline{K}$,
and $(\Gamma \backslash M,dt^2)$ is a connected Riemannian nilmanifold.  
Theorem \ref{coverings} says that this construction is exhaustive.}
\hfill $\diamondsuit$
\end{remark}

Another consequence is that coverings of our class of homogeneous metric 
space quotients, in particular of Riemannian nilmanifolds, always
are normal coverings.

\begin{corollary}\label{normal}
Let $\varphi: (M_1,d_1) \to (M_2,d_2)$ be a locally isometric 
covering space in which $(M_2,d_2)$ is a metric space on which a 
Lie group $S/\Delta_2$\,, $S$ exponential solvable and $\Delta_2$ discrete and
central in $S$, acts effectively and transitively 
by isometries.  Then $(M_1,d_1)$ is a metric space on which another 
quotient group $S/\Delta_1$ acts effectively and transitively 
by isometries.  Further, $\Delta_1 \subset \Delta_2$\,, and the covering
is normal with deck transformation group $\Delta_2/\Delta_1$\,.
\end{corollary}

\begin{proof} As described above, the universal covering
$\psi_2: (M,d) \to (M_2,d_2)$ is given by dividing out with a
discrete subgroup $\Gamma_2$ of the center of $S$.  As
$\varphi: (M_1,d_1) \to (M_2,d_2)$ is a locally isometric covering,
the universal covering $\psi_1: (M,d) \to (M_1,d_1)$ is given by 
dividing out with a subgroup $\Gamma_1$ of $\Gamma_2$\,.  Since the 
center of $S$ is abelian, $\Gamma_1$ is normal in $\Gamma_2$, so 
$\varphi$ is the normal locally isometric covering given by dividing 
out with $\Gamma_2/\Gamma_1$.
\end{proof}

In the case where $(M,d)$ is a Riemannian manifold $(M,ds^2)$ or a
Finsler manifold $(M,F)$,
every $\xi \in \gg$ defines a Killing vector field $\xi^M$ on $(M,d)$.
If $\xi^M$ has bounded length on $(M,d)$ then $\exp(t\xi)$ is a bounded
isometry for all real $t$.  Now Theorem \ref{central} implies

\begin{corollary}\label{bded-killing}
Suppose that the metric space $(M,d)$ is Riemannian $($or Finsler$)$.
Let $M = G/K = SK/K$ where $G = I(M,d)$, $S$ is an exponential solvable 
Lie group acting transitively on $(M,d)$, and $K$ is an isotropy subgroup.  
Let $\xi \in \gg$ such that $\xi^M$ is a Killing vector field of bounded
length on $(M,d)$.  Then $\xi$ belongs to the center of $\gs$ and 
$\xi^M$ has constant length on $(M,d)$.
\end{corollary}

In the Riemannian nilmanifold setting, the formulation of Theorem
\ref{coverings} is a bit less complicated.  It becomes

\begin{corollary}\label{nil-coverings}
Let $(M,ds^2)$ be a simply connected Riemannian nilmanifold.  Consider
a Riemannian covering $\psi: (M,ds^2) \to \Gamma \backslash (M,ds^2)$.
Then the following conditions are equivalent.
\smallskip

{\rm 1.} $\Gamma$ consists of bounded isometries of $(M,ds^2)$.
\smallskip

{\rm 2.} $\Gamma$ consists of CW isometries of $(M,ds^2)$.
\smallskip

{\rm 3.} $G = I(M,ds^2) = N\rtimes K$ semidirect product with $N$ nilpotent,
and the group $\Gamma$ is central in $N$.
\smallskip

{\rm 4.} $\Gamma \backslash (M,ds^2)$ is a homogeneous Riemannian manifold.
\smallskip

{\rm 5.} $\Gamma \backslash (M,ds^2)$ is a Riemannian nilmanifold.
\smallskip

\noindent
Further, every connected Riemannian nilmanifold is isometric to a manifold
$\Gamma \backslash (M,ds^2)$ as just described.
\end{corollary}

\bigskip
\noindent Department of Mathematics \hfill\newline
\noindent University of California\hfill\newline
\noindent Berkeley, California 94720--3840, USA\hfill\newline
\smallskip
\noindent {\tt jawolf@math.berkeley.edu}

\enddocument
\end